\documentclass[12pt]{article}

\usepackage[margin=1in]{geometry}
\usepackage{amsmath}
\usepackage{amssymb}
\usepackage{amsthm}
\usepackage{graphicx}
\usepackage{algorithm}
\usepackage{algorithmic}

\usepackage{bbm,bm}
\usepackage{authblk}
\usepackage{cancel}
\usepackage[round,sort]{natbib}
\usepackage{multirow}
\usepackage{hyperref}

\def\PP{\mathbb{P}}
\def\RR{\mathbb{R}}
\def\ZZ{\mathbb{Z}}

\DeclareMathOperator{\diag}{diag}
\DeclareMathOperator{\sign}{sign}
\DeclareMathOperator{\Adj}{Adj}

\def\U{{\cal U}}

\def\N{{\cal N}}

\newtheorem{theorem}{Theorem}

\begin{document}

\title{Uniform Sampling from\\High-dimensional Spectral Norm Balls}
\author{Michael R. Metel\thanks{michaelrmetel@gmail.com}}
\maketitle	

\begin{abstract}
\noindent Motivated by an application in machine learning optimization, this paper focuses on the challenges of sampling a matrix uniformly from the unit spectral norm ball. It is proven that all singular values of sampled matrices converge to 1 almost surely as the matrix dimensions increase. This result provides the theoretical justification for a proposed simple sampling method applicable for large dimension sizes matching matrices found in modern large language models. Experimental results demonstrate both the convergence of the singular values, as well as the exact and proposed approximate sampling methods.
\end{abstract}

\section{Introduction}
\label{intro}
This paper studies the uniform sampling of matrices $\bm{\Delta}\in\RR^{n\times m}$ from the unit spectral norm ball,
\begin{alignat}{6}
B:=\{\bm{\Delta}\in \RR^{n\times m}: \sigma_{max}(\bm{\Delta})\leq 1\},\nonumber
\end{alignat}
where $\sigma_{max}(\bm{\Delta})$ denotes the largest singular value of $\bm{\Delta}$. 
A sample from this distribution will be denoted as $\hat{\bm{\Delta}}\sim\U(B)$. Given that a sample from the spectral norm ball of radius $\rho>0$ can be generated as $\rho\hat{\bm{\Delta}}$, the focus on $\rho=1$ is without loss of generality. After a possible transpose, it will also be assumed that $m\geq n$.  

A polynomial-time recipe to sample from this distribution can be found across two papers \citep{calafiore2000,calafiore2002}, which was motivated by probabilistic robustness techniques for control systems. The idea is to consider a normalized singular value decomposition of $\bm{\Delta}$,  
\begin{alignat}{6}
	&\bm{\Delta}=\bm{U}\bm{\Sigma}\bm{V}^T,\nonumber
\end{alignat}
where $\bm{U}\in\RR^{n,n}$ and $\bm{V}\in\RR^{m,n}$ have orthonormal columns with $V_{1,j}>0$ (almost surely) for $j=1,...,n$, and $\bm{\Sigma}=\diag(\sigma_1,..,\sigma_n)$ for singular values $\sigma_1\geq\sigma_2\geq ...\geq \sigma_n$. It is proven \citep[Theorem 2]{calafiore2000} that to uniformly sample over $B$, $\bm{U}$, $\bm{\Sigma}$, and $\bm{V}$ can be sampled independently, with $\bm{U}$ and $\bm{V}$ sampled uniformly over the orthogonal and normalized semi-orthogonal matrices, respectively, which can be done using samples from the standard Normal distribution and QR decomposition as shown in Algorithm \ref{alg:UV}.
\begin{algorithm}[H]
	\caption{Uniformly sampling $\bm{U}$ and $\bm{V}$}
	\label{alg:UV}
	\begin{algorithmic}			
		\STATE Sample $\hat{\bm{U}}\in\RR^{n,n}$ and $\hat{\bm{V}}\in\RR^{m,m}$: $\hat{U}_{j,k}\sim \N(0,1)$ and $\hat{V}_{j,k}\sim \N(0,1)$ $\forall j,k$
        \STATE $[\bm{Q}(\hat{\bm{U}}),\bm{R}(\hat{\bm{U}})]=\text{QR}(\hat{\bm{U}})$ 
        \STATE $[\bm{Q}(\hat{\bm{V}}),\bm{R}(\hat{\bm{V}})]=\text{QR}(\hat{\bm{V}})$
	    \STATE $\hat{\bm{U}}=\bm{Q}(\hat{\bm{U}})\diag(\sign(R(\hat{\bm{U}})_{11},...,\sign(R(\hat{\bm{U}})_{nn})$
        \STATE $\hat{\bm{V}}=\bm{Q}(\hat{\bm{V}})[:,1:n]$ 
        \STATE $\hat{\bm{V}}=\hat{\bm{V}}\diag(\sign(\hat{V}_{11}),...,\sign(\hat{V}_{1n}))$
		\STATE {\bfseries Output:} $\hat{\bm{U}}$,$\hat{\bm{V}}$
	\end{algorithmic}
\end{algorithm}

\noindent In order to sample $\bm{\Sigma}$ requires significantly more effort. Singular values are proposed to be sampled iteratively following the conditional distribution method \citep[Chapter 11]{devroye1986}. The algorithm presented in \citep[page 1269]{calafiore2002} samples singular values using exact conditional densities which contain both constants and exponents which result in over and underflows for modest problem sizes. As the motivation is to sample matrices at the scale of modern large language models, e.g., the largest linear layer weight matrix of the 32 billion-parameter Qwen 3 model \citep{qwen3} is $n\times m = 5,120\times 25,600$, the first step, detailed in Section \ref{sva}, was to simplify the algorithm by using functions which are only proportional to the exact conditional densities.

After failing to scale up the algorithm directly, the distribution of the singular values as the dimensions of $\bm{\Delta}$ increase is studied in Section \ref{theo}. The main contribution presented as Theorem \ref{theorem1} is that all singular values converge to 1 almost surely as $n\rightarrow \infty$, and in Theorem \ref{theorem2}, it is proven that all singular values converge to 1 almost surely as only $m\rightarrow \infty$ for a fixed value of $n$. 

These results present a viable approximate sampling algorithm consisting of setting $\bm{\Sigma}=\bm{I}$, and simply using Algorithm \ref{alg:UV} to sample from $\U(B)$. Section \ref{Num_exp} contains further implementation details of the exact sampling algorithm, and empirical results demonstrating the convergence of the exact and proposed approximate sampling methods. 

There appears to have been no attempts to develop a practical approach to sample from $\U(B)$ at the scale of real-world problems. The software package RACT \citep{tremba2008} supports sampling from $\U(B)$, but a warning stating ``This can take a LONG time..." will be given when $\min(m,n)>6$. The experiments presented in \citep{calafiore2002} are also restricted to problem sizes upper bounded by $\max(n,m)\leq 4$. This section ends with an application that motivated this work, requiring an efficient sampling approach for high-dimensional matrices. 

\subsection{Neural Network Training with Random Weight Perturbation}
\label{application}

Gaussian, uniform $l_2$, and $l_\infty$-norm ball perturbations have been applied to loss function decision variables when computing stochastic gradients for first-order optimization methods \citep[Appendix E]{duchi2012}. This has a smoothing effect, resulting in non-differentiable Lipchitz continuous functions becoming smooth in expectation, enabling provably convergent optimization algorithms. From a neural network perspective, random weight perturbation pushes convergence to flatter local minima, improving a trained model's ability to generalize \citep{bisla2022}. 
These forms of perturbation implicitly view a model's weights as a vector $\bm{w}\in\RR^N$. Considering that neural networks contain several layers, each of which consisting of matrices, a line of research \citep{bernstein2025} has focused on incorporating this structure into the design of optimization methods. Letting each weight matrix $\bm{W}^i$ live in its own vector space $V^i$ with norm $\|\cdot\|_i$, the neural network is contained in the product space $\Pi_iV^i$, with norm $\max\limits_ia_i\|\bm{W}^i\|_i$ for per-matrix weights $a_i>0$. Assuming that the input $x^i\in\RR^m$ and output $y^i\in\RR^n$ of each linear layer lie in Euclidean space, $\|\cdot\|_i$ is set to the induced matrix norm of $\bm{W}^i$: the spectral norm. Given an efficient method to sample $\hat{\bm{\Delta}}^i\sim\U(B)$, $\bm{W}^i$ can then be perturbed by $\rho_i\hat{\bm{\Delta}}^i$ for per-matrix radii $\rho_i$, enabling the extension of random weight perturbation to architecture-aware optimization algorithms.

\section{Singular Value Sampling Algorithm}
\label{sva}

This section describes an implementation of the algorithm described in \citep[page 1269]{calafiore2002} to sample iteratively the singular values $\sigma_1\geq \sigma_2\geq...\geq \sigma_n$. While the differences between the described implementation and the published method are described, for a full derivation of this algorithm, it is recommended to consult the original publication. 

\subsection{Definitions}
\label{notdef}
Let $v:=\frac{1}{2}(m-n-1)$, $\beta:=\frac{n}{2}(n-1)$,  
    $\bm{J}:=\begin{bmatrix}
    0 & 1\\    
    -1 & 0\\
    \end{bmatrix}$, and $\bm{1}_n:=[1,1,...,1]^T\in\RR^n$.
Let $\bm{S}\in \RR^{n\times n}$ be a skew-symmetric matrix, where $S_{ij}:=\frac{i-j}{(i+v)(j+v)(i+j+2v)}$, let $\bm{F}\in\RR^n$ be a vector defined as $F_j:=\frac{1}{j+v}$, and let $\bm{X}: \RR\rightarrow \RR^n$ be defined as $X_j(x):=x^{j-1}$. When sampling $\sigma_i$, the algorithm depends on the parity of $n$ and $i$, so the resulting four cases will need to be considered. The functions defined below are used for the recursive definitions of the marginal densities of the singular values.\\

\noindent\underline{Even number of rows $n$}: 
    Let 
    $$\textbf{Z}_0(x):=\diag(x^{\beta},x^{\beta-1},...,x^{\beta-n-1})\bm{Q}_0^{-1}\diag(x^{\beta},x^{\beta-1},...,x^{\beta-n-1}),$$ or element-wise, 
    $(Z_0(x))_{ij}=x^{2\beta-i-j+2}(Q_0^{-1})_{ij}$,     
    where $\bm{Q}_0:=\bm{S}$, and let $\bm{h}_e(x):=\bm{X}(x)$. The functions $\bm{Z}_0$ and $\bm{h}_e$ are used to generate even-indexed samples $\hat{\sigma}_i$. Let 
    $$\bm{Z}_1(x):=\diag(x^{\beta},x^{\beta-1},...,x^{\beta-n-1},x^{\beta},x^{\beta})\bm{Q}_1^{-1}\diag(x^{\beta},x^{\beta-1},...,x^{\beta-n-1},x^{\beta},x^{\beta}),$$ 
    or
    \begin{alignat}{6}
    (Z_1(x))_{ij}:=\begin{cases}
    x^{2\beta-i-j+2}(Q_1^{-1})_{ij} & \text{if } i,j\in[1,2,...,n] \\
    x^{2\beta-i+1}(Q_1^{-1})_{ij} & \text{if } i\in [1,2,...,n] \text{ and } j\in[n+1,n+2] \\
    x^{2\beta-j+1}(Q_1^{-1})_{ij} & \text{if } j\in [1,2,...,n] \text{ and } i\in[n+1,n+2] \\
    x^{2\beta}(Q_1^{-1})_{ij} & \text{if } i,j\in [n+1,n+2],
    \end{cases}\label{Z_1_even}
    \end{alignat} 
    where 
    $$\bm{Q}_1:=\begin{bmatrix}
    \bm{S} & \bm{F} & \bm{1}_n\\                      
    -\bm{F}^T & 0 & 0\\
    -\bm{1}_n^T & 0 & 0\\
    \end{bmatrix},$$
    and let $\bm{h}_o(x):=[\bm{X}(x);0;0]$, which are used to generate odd-indexed samples $\hat{\sigma}_i$.\\

\noindent\underline{Odd number of rows $n$}: 
    Let 
    $$\bm{Z}_0(x):=\diag(x^{\beta},x^{\beta-1},...,x^{\beta-n-1},x^{\beta})\bm{Q}_0^{-1}\diag(x^{\beta},x^{\beta-1},...,x^{\beta-n-1},x^{\beta}),$$     
    where 
    $$\bm{Q}_0:=\begin{bmatrix}
    \bm{S} & \bm{F}\\
    -\bm{F}^T & 0\\    
    \end{bmatrix},$$  
    $$\bm{Z}_1(x):=\diag(x^{\beta},x^{\beta-1},...,x^{\beta-n-1},x^{\beta})\bm{Q}_1^{-1}\diag(x^{\beta},x^{\beta-1},...,x^{\beta-n-1},x^{\beta}),$$    
    where 
    $$\bm{Q}_1:=\begin{bmatrix}
    \bm{S} & \bm{1}_n\\    
    -\bm{1}_n^T & 0\\
    \end{bmatrix},$$
    and $\bm{h}_e(x):=\bm{h}_o(x):=[\bm{X}(x);0]$. The exponent of $x$ in $(Z_0(x))_{ij}$ and $(Z_1(x))_{ij}$ for all $1\leq i,j\leq n+1$ is identical to \eqref{Z_1_even}.

In the original algorithm, $\bm{Z}_0$ and $\bm{Z}_1$ contain the matrices $\Adj(\bm{Q}_0)$ and $\Adj(\bm{Q}_1)$ instead of $\bm{Q}_0^{-1}$ and $\bm{Q}_1^{-1}$, with our implementation not including the constants $\det(\bm{Q}_0)$ and $\det(\bm{Q}_1)$, which both underflow when $n\geq 28$ using 64-bit floating point arithmetic\footnote{This was tested using numpy.linalg.det in the same Python environment described in Section \ref{Num_exp}.}. Let the scalar function 
$$p(x):=x^{\beta}$$
also be defined, which will replace the scalar functions $p_0(x)=x^{\beta}\sqrt{\det(\bm{Q}_0)}$ and $p_1(x)=x^{\beta}\sqrt{\det(\bm{Q}_1)}$ used in the original algorithm. The next section verifies that these simplifications preserve the proportionality to the conditional probability distribution of all $\sigma_i$. 

\subsection{Algorithm Implementation}
\label{algorithm}
With the given definitions, the basic algorithm is shown in Algorithm \ref{alg:sigma}.\\ 

\noindent Line \ref{ln:1}: The probability density function of $\sigma_1$ is equal to $f_1(\sigma_1)=\frac{K}{2^{n-1}}\sigma_1^{2(v+1)(n-1)}p_1(\sigma_1^2)\sigma_1^{m-n}$, where $K$ is a normalization constant \citep[Eq. 5]{calafiore2002}. Removing constants and simplifying the exponents results in the given function proportional to $f_1$.\\

\noindent Lines \ref{ln:3}, \ref{ln:5}, \& \ref{ln:52}: These are exponents that $\sigma_i$ will be taken to in its conditional density. Simplifying, $e_1(i)=(m-n)(n-i+1)+n-i$, and by considering the case $m=n$, $e_1(i)\geq n-i$. When $i$ is even, $e_2(i)=-(i-1)n(n-1)$, and when odd, $e_2(i)=-(i-2)n(n-1)$. The exponent $e_1(i)+e_2(i)$ is decreasing in $i$: For even $i$, $e_1(i+1)+e_2(i+1)-(e_1(i)+e_2(i))=-(m-n)-1$, and for odd $i$, $e_1(i+1)+e_2(i+1)-(e_1(i)+e_2(i))=-(m-n)-1-2n(n-1)$.\\

\noindent Line \ref{rec:2}: This is the recursive calculation of $\bm{Z}_i(x)$ as given in \citep[page 1269]{calafiore2002}. Defining $\tilde{\bm{Z}}_0(x):=\det(\bm{Q}_0)\bm{Z}_0(x)$ and $\tilde{\bm{Z}}_1(x):=\det(\bm{Q}_1)\bm{Z}_1(x)$ as the original definitions of $\bm{Z}_0$ and $\bm{Z}_1$, and using the same recursion to compute $\tilde{\bm{Z}}_i$,   $\tilde{\bm{Z}}_i(x)=\det(\bm{Q}_0)^{2^\frac{i}{2}} \bm{Z}_i(x)$ when $i$ is even, and $\tilde{\bm{Z}}_i(x)=\det(\bm{Q}_1)^{2^\frac{i-1}{2}} \bm{Z}_i(x)$ when $i$ is odd.\\

\noindent Lines \ref{sam:1} \& \ref{sam:2}: The marginal density of $\{\sigma_j\}_{j=1}^i$ is equal to 
$$\frac{K}{2^{n-i}}p_i(\sigma_i^2)\sigma_i^{2(v+1)(n-i)+m-n}\Pi_{j=1}^{i-1}\sigma_j^{m-n},$$
where $p_i(x)=p_0(x)^{-(i-1)}|\bm{h}_e(\sigma_{i-1}^2)^T\tilde{\bm{Z}}_{i-2}(x)\bm{h}_e(x)|$ when $i$ is even and 

\noindent$p_i(x)=p_1(x)^{-(i-2)}|\bm{h}_o(\sigma_{i-2}^2)^T\tilde{\bm{Z}}_{i-2}(x)\bm{h}_o(\sigma_{i-1}^2)|$ when $i$ is odd. Given that $\{\sigma_j=\hat{\sigma}_j\}_{j=1}^{i-1}$ are known, the required conditional density of $\sigma_i$ is proportional to this marginal density. By removing the constant $\frac{K}{2^{n-i}}\Pi_{j=1}^{i-1}\sigma_j^{m-n}$, and recalling that $p_0(x) \text{ \& }p_1(x)\propto x^{\beta}$ and $\bm{Z}_{i-2}\propto\tilde{\bm{Z}}_{i-2}$, the functions on Lines \ref{sam:1} \& \ref{sam:2} are proportional to $f_i(\sigma_i|\sigma_1=\hat{\sigma}_1,...,\sigma_{i-1}=\hat{\sigma}_{i-1})$. 

\begin{algorithm}
	\caption{Sampling singular values $\sigma_1\geq \sigma_2\geq...\geq \sigma_n$}
	\label{alg:sigma}
	\begin{algorithmic}[1]
		\STATE Generate sample $\hat{\sigma}_1$ from density $f_1(\sigma_1)\propto\sigma_1^{mn-1}$\label{ln:1}  
        \FOR{$i=2,...,n$}
        \STATE {$e_1(i)=2(v+1)(n-i)+m-n$}\label{ln:3}   
        \IF{$i\mod 2 = 0$} 
        \STATE {$e_2(i)=-2(i-1)\beta$}\label{ln:5}   
        \STATE $d_i(x)=\bm{h}_e(\hat{\sigma}^2_{i-1})^T\bm{Z}_{i-2}(x)\bm{h}_e(x)$\label{d:1} 
        \STATE Generate sample $\hat{\sigma}_i$ from density $f_i(\sigma_i|\sigma_1=\hat{\sigma}_1,...,\sigma_{i-1}=\hat{\sigma}_{i-1})\propto\sigma_i^{e_1(i)+e_2(i)}|d_i(\sigma_i^2)|$\label{sam:1} 
        \STATE Update $d_i(x)=\bm{h}_e(\hat{\sigma}^2_{i-1})^T\bm{Z}_{i-2}(x)\bm{h}_e(\hat{\sigma}_i^2)$
        \STATE $\bm{H}(i)=[\bm{h}_e(\hat{\sigma}^2_{i-1}),\bm{h}_e(\hat{\sigma}^2_{i})]\bm{J}[\bm{h}_e(\hat{\sigma}^2_{i-1}),\bm{h}_e(\hat{\sigma}^2_{i})]^T$      
        \ELSE       
        \STATE {$e_2(i)=-2(i-2)\beta$}\label{ln:52}    
        \STATE $d_i(x)=\bm{h}_o(\hat{\sigma}^2_{i-2})^T\bm{Z}_{i-2}(x)\bm{h}_o(\hat{\sigma}_{i-1}^2)$\label{d:2}
        \STATE Generate sample $\hat{\sigma}_i$ from density $f_i(\sigma_i|\sigma_1=\hat{\sigma}_1,...,\sigma_{i-1}=\hat{\sigma}_{i-1})\propto\sigma_i^{e_1(i)+e_2(i)}|d_i(\sigma_i^2)|$\label{sam:2}      
        \STATE $\bm{H}(i)=[\bm{h}_o(\hat{\sigma}^2_{i-2}),\bm{h}_o(\hat{\sigma}^2_{i-1})]\bm{J}[\bm{h}_o(\hat{\sigma}^2_{i-2}),\bm{h}_o(\hat{\sigma}^2_{i-1})]^T$           
        \ENDIF        
        \STATE $\bm{Z}_i(x)=d_i(x)\bm{Z}_{i-2}(x)+\bm{Z}_{i-2}(x)\bm{H}(i)\bm{Z}_{i-2}(x)$\label{rec:2}        
        \ENDFOR
	\end{algorithmic}
\end{algorithm}

\section{Convergence of $\{\sigma_i\}$ in High Dimension}
\label{theo}

The main result of this work is that as $n$ increases, all singular values converge to $1$ almost surely, where we recall that $\sigma_1\geq \sigma_2\geq ...\geq \sigma_n$ by definition.
\begin{theorem}
\label{theorem1}
For matrices uniformly distributed in the unit spectral norm ball, $\bm{\Delta}\sim\U(B)$,  $\lim\limits_{n\rightarrow \infty}\sigma_n=1$ almost surely.
\end{theorem}

\begin{proof}
From the definitions of $\bm{Z}_0$ and $\bm{Z}_1$, the recursions defining $\bm{Z}_i$ for $i>1$, and the definitions of $d_i$, all elements of $\bm{Z}_i$ and the functions $d_i$ are polynomials in $x$, with all $\sigma_i$ sampled from univariate polynomial functions. Let the unnormalized conditional density function of $\sigma_i$, $\tilde{f}_i(\sigma_i|\sigma_1,...,\sigma_{i-1{}}):=\sigma_i^{e_1(i)+e_2(i)}|d_i(\sigma_i^2)|$, with domain $[0,\sigma_{i-1}]$, be written as  
\begin{alignat}{6}
\sum_{j}a_{ij}\sigma_i^{t_{ij}},\nonumber
\end{alignat}
where all coefficients $a_{ij}>0$. The polynomial changes with $n$, but as will be seen, the result is independent of the number of polynomial terms or the values of their coefficients. By normalization, $f_i(\sigma_i|\sigma_1,...,\sigma_{i-1})$ is equal to  
\begin{alignat}{6}
\frac{\sum_{j}a_{ij}\sigma_i^{t_{ij}}}{\sum_{j}\frac{a_{ij}}{t_{ij}+1}\sigma_{i-1}^{t_{ij}+1}}.\nonumber
\end{alignat} 
Viewing $f_i(\sigma_i|\sigma_1,...,\sigma_{i-1})$ as a weighted average of the terms
\begin{alignat}{6}
\frac{a_{ij}\sigma_i^{t_{ij}}}{\frac{a_{ij}}{t_{ij}+1}\sigma_{i-1}^{t_{ij}+1}}=\frac{t_{ij}+1}{\sigma_{i-1}}(\frac{\sigma_i}{\sigma_{i-1}})^{t_{ij}}\nonumber
\end{alignat} 
with weights 
\begin{alignat}{6}
\frac{\frac{a_{ij}}{t_{ij}+1}\sigma_{i-1}^{t_{ij}+1}}{\sum_{k}\frac{a_{ik}}{t_{ik}+1}\sigma_{i-1}^{t_{ik}+1}},\nonumber
\end{alignat} 
$f_i(\sigma_i|\sigma_1,...,\sigma_{i-1})$ can be bounded as 
\begin{alignat}{6}
\inf\limits_j\frac{t_{ij}+1}{\sigma_{i-1}}(\frac{\sigma_i}{\sigma_{i-1}})^{t_{ij}}\leq f_i(\sigma_i|\sigma_1,...,\sigma_{i-1})\leq \sup\limits_j\frac{t_{ij}+1}{\sigma_{i-1}}(\frac{\sigma_i}{\sigma_{i-1}})^{t_{ij}}.\label{dist_ineq}
\end{alignat} 

\noindent The task now is to lower bound the exponents $\{t_{ij}\}_j$ for each $i$. By showing that these exponents are growing sufficiently fast with $n$, the theorem can be proven. To first compute the entries of 
\begin{alignat}{6}
\bm{Z}_i(x)&=d_i(x)\bm{Z}_{i-2}(x)+\bm{Z}_{i-2}(x)\bm{H}(i)\bm{Z}_{i-2}(x),\nonumber
\end{alignat}
the values of $d_i$ are derived: When $i$ and $n$ are even, 
\begin{alignat}{6}
	d_i(x)&=\bm{X}(\sigma_{i-1}^2)^T\bm{Z}_{i-2}(x)\bm{X}(x)=[1,\sigma_{i-1}^2,...,\sigma_{i-1}^{2(n-1)}]\bm{Z}_{i-2}(x)[1,x,...,x^{n-1}]^T\nonumber\\    
    &=\sum_{p=1}^{n}\sum_{q=1}^{n}\sigma_{i-1}^{2(p-1)}x^{q-1}(\bm{Z}_{i-2}(x))_{p,q},\nonumber
\end{alignat}
and when $n$ is odd, 
\begin{alignat}{6}
	d_i(x)&=[\bm{X}(\sigma_{i-1}^2)^T,0]\bm{Z}_{i-2}(x)[\bm{X}(x);0]=[1,\sigma_{i-1}^2,...,\sigma_{i-1}^{2(n-1)},0]\bm{Z}_{i-2}(x)[1,x,...,x^{n-1},0]^T\nonumber\\
    &=\sum_{p=1}^{n}\sum_{q=1}^{n}\sigma_{i-1}^{2(p-1)}x^{q-1}(\bm{Z}_{i-2}(x))_{p,q}.\nonumber
\end{alignat}
When $i$ is odd and $n$ is even,
\begin{alignat}{6}
	d_i(x)&=[\bm{X}(\sigma_{i-2}^2)^T,0,0]\bm{Z}_{i-2}(x)[\bm{X}(\sigma_{i-1}^2);0;0]=\sum_{p=1}^{n}\sum_{q=1}^{n}\sigma_{i-2}^{2(p-1)}\sigma_{i-1}^{2(q-1)}(\bm{Z}_{i-2}(x))_{p,q},\nonumber
\end{alignat}
and when $n$ is odd,
\begin{alignat}{6}
	d_i(x)=&[\bm{X}(\sigma_{i-2}^2)^T,0]\bm{Z}_{i-2}(x)[\bm{X}(\sigma_{i-1}^2);0]=\sum_{p=1}^{n}\sum_{q=1}^{n}\sigma_{i-2}^{2(p-1)}\sigma_{i-1}^{2(q-1)}(\bm{Z}_{i-2}(x))_{p,q}.\nonumber
\end{alignat}

\noindent For the entries of $\bm{Z}_i$, when $i$ is even,
\begin{alignat}{6}
(\bm{Z}_i(x))_{jk}=(&\sum_{p=1}^{n}\sum_{q=1}^{n}\sigma_{i-1}^{2(p-1)}x^{q-1}(\bm{Z}_{i-2}(x))_{p,q})(\bm{Z}_{i-2}(x))_{jk}\label{Z_recurs_1}\\
+&\sum_{p=1}^n\sum_{q=1}^n(\bm{Z}_{i-2}(x))_{jp}\bm{H}(i)_{pq}(\bm{Z}_{i-2}(x))_{qk},\nonumber
\end{alignat}
and when $i$ is odd,
\begin{alignat}{6}
(\bm{Z}_i(x))_{jk}=(&\sum_{p=1}^{n}\sum_{q=1}^{n}\sigma_{i-2}^{2(p-1)}\sigma_{i-1}^{2(q-1)}(\bm{Z}_{i-2}(x))_{p,q})(\bm{Z}_{i-2}(x))_{jk}\label{Z_recurs_2}\\
+&\sum_{p=1}^n\sum_{q=1}^n(\bm{Z}_{i-2}(x))_{jp}\bm{H}(i)_{pq}(\bm{Z}_{i-2}(x))_{qk},\nonumber
\end{alignat}
where the summations involving $\bm{H}(i)$ only go up to $n$ given that $H(i)_{pq}=0$ when $p$ or $q>n$.
The matrices $\bm{Q}_0^{-1}$ and $\bm{Q}_1^{-1}$ in the definitions of $\bm{Z}_0(x)$ and $\bm{Z}_1(X)$, and the matrices $\bm{H}(i)$ are skew-symmetric, hence zero-diagonal. For simplicity, this will be ignored by lower bounding all exponents $\{t_{ij}\}_j$ of each $\sigma_i$, even those whose corresponding coefficient $a_{ij}=0$.   

Let $E(i)_{jk}$ be the infimum of the exponents of $x$ in the polynomial $(Z_i(x))_{jk}$, and 
$E(i):=\inf\limits_{j,k}E(i)_{jk}$. From Section \ref{notdef}, it can be observed that $E(0)=E(0)_{nn}$ and $E(1)=E(1)_{nn}$ 
for even or odd $n$, with the equalities holding uniquely for $(j,k)=(n,n)$ when $n\geq 2$. From the above recursions, for $i>1$, $E(i)=E(i)_{nn}$ and $E(i)=2E(i-2)$ for both even and odd $i$, by considering the term $(Z_{i-2}(x))_{nn}H(i)_{nn}(Z_{i-2}(x))_{nn}$ in \eqref{Z_recurs_1} and \eqref{Z_recurs_2}, and also $\sigma_{i-2}^{2(n-1)}\sigma_{i-1}^{2(n-1)}(Z_{i-2}(x))_{n,n}(Z_{i-2}(x))_{nn}$ in \eqref{Z_recurs_2}. 

Let $D(i)$ be a lower bound on the exponents of $x$ in the polynomial $d_i(x)$. By considering odd $i$ in particular, where $x$ is only an argument of $\bm{Z}_{i-2}$, setting $D(i)=E(i-2)$ for all $i$ is valid, where $D(2)=D(3)=2\beta-2n+2=n^2-3n+2$,
and in general, $D(i)=2^{\frac{i-2}{2}}(n^2-3n+2)$ and $D(i)=2^{\frac{i-3}{2}}(n^2-3n+2)$ for all even and odd $i>1$, respectively.

These exponents are with respect to the argument $x$, whereas $d_i$ is evaluated at  $\sigma_i^2$, so by adding $2D(i)$ to the lower bound of $e_1(i)+e_2(i)$, by considering $m=n$ from Section \ref{algorithm}, let $e_{\tilde{f}}(i)$ be a lower bound on the exponents of $\sigma_i$ contained in $\tilde{f}_i$ for $i>1$, where for even $i$, 
\begin{alignat}{6}
e_{\tilde{f}}(i)&=n-i-(i-1)n(n-1)+2*2^{\frac{i-2}{2}}(n^2-3n+2)\label{exp_lb_even}\\
&=(2^{\frac{i}{2}}-i+1)n^2+(i-3*2^{\frac{i}{2}})n+2^{\frac{i+2}{2}}-i,\nonumber
\end{alignat}
and for odd $i$,
\begin{alignat}{6}
e_{\tilde{f}}(i)&=n-i-(i-2)n(n-1)+2*2^{\frac{i-3}{2}}(n^2-3n+2)\label{exp_lb_odd}\\
&=(2^{\frac{i-1}{2}}-i+2)n^2+(i-3*2^{\frac{i-1}{2}}-1)n+2^{\frac{i+1}{2}}-i.\nonumber
\end{alignat}

\noindent The function $e_{\tilde{f}}$ can be further lower bounded by $c2^{\frac{i}{2}}n^2$ and $c2^{\frac{i-1}{2}}n^2$ for even and odd $i$ using a constant $0<c<\frac{1}{4}$ for $n$ sufficiently large: For even $i$,
\begin{alignat}{6}
\frac{e_{\tilde{f}}(i)}{2^{\frac{i}{2}}n^2}&=\frac{(2^{\frac{i}{2}}-i+1)n^2+(i-3*2^{\frac{i}{2}})n+2^{\frac{i+2}{2}}-i}{2^{\frac{i}{2}}n^2}\nonumber\\
&=\frac{(2^{\frac{i}{2}}-i+1)}{2^{\frac{i}{2}}}+\frac{(i-3*2^{\frac{i}{2}})}{2^{\frac{i}{2}}n}+\frac{2^{\frac{i+2}{2}}-i}{2^{\frac{i}{2}}n^2}\nonumber\\
&\geq \frac{1}{4} -\frac{3}{n}+0,\nonumber
\end{alignat}
where the first term equals $\frac{1}{4}$ when $i=4$. For odd $i$, 
\begin{alignat}{6}
\frac{e_{\tilde{f}}(i)}{2^{\frac{i-1}{2}}n^2}&=\frac{(2^{\frac{i-1}{2}}-i+2)n^2+(i-3*2^{\frac{i-1}{2}}-1)n+2^{\frac{i+1}{2}}-i}{2^{\frac{i-1}{2}}n^2}\nonumber\\
&\geq \frac{1}{4} -\frac{3}{n}+0,\nonumber 
\end{alignat}
where the bound $c2^{\frac{i-1}{2}}n^2<\frac{1}{4}2^{\frac{i-1}{2}}n^2$ is also valid for $i=1$: $mn-1\geq n^2-1>\frac{n^2}{4}$ when $n\geq 2$. For both even and odd $i$, let $n_c\geq 2$ be sufficiently large such that $e_{\tilde{f}}(i)\geq c2^{\frac{i-1}{2}}n^2$ holds for $i\geq 1$ and $n\geq n_c$ for a given $c>0$.

Given that a lower bound on $\{t_{ij}\}_j$ for each $i$ has been established, the next step is to upper bound $\PP(1-\sigma_n>\epsilon)$ for any $\epsilon>0$. Defining $\sigma_0:=1$,

\begin{alignat}{6}
\PP(1-\sigma_n>\epsilon)&=\PP(\sum_{i=1}^{n}\sigma_{i-1}-\sigma_{i}>\epsilon)\nonumber\\
&\leq\PP(\cup_{i=1}^{n}(\sigma_{i-1}-\sigma_i>\frac{\epsilon}{n}))\nonumber\\
&\leq\sum_{i=1}^{n}\PP(\sigma_{i-1}-\sigma_i>\frac{\epsilon}{n})\text{, where}  \label{part_1} 
\end{alignat}
\begin{alignat}{6}
\PP(\sigma_{i-1}-\sigma_i>\frac{\epsilon}{n})&=\int_0^1\cdot\cdot\cdot\int_0^{\sigma_{i-2}}\int_0^{\sigma_{i-1}-\frac{\epsilon}{n}}f_i(\sigma_1,...,\sigma_i)d\sigma_id\sigma_{i-1}...d\sigma_1\nonumber\\
&=\int_0^1\cdot\cdot\cdot\int_0^{\sigma_{i-2}}\int_0^{\sigma_{i-1}-\frac{\epsilon}{n}}f_i(\sigma_i|\sigma_1,...,\sigma_{i-1})d\sigma_if_{i-1}(\sigma_1,...,\sigma_{i-1})d\sigma_{i-1}...d\sigma_1\text{, and}\label{cond_prob}
\end{alignat}
\begin{alignat}{6}
\int_0^{\sigma_{i-1}-\frac{\epsilon}{n}}f_i(\sigma_i|\sigma_1,...,\sigma_{i-1})d\sigma_i&\leq\int_0^{\sigma_{i-1}-\frac{\epsilon}{n}}\sup\limits_{j}\frac{t_{ij}+1}{\sigma_{i-1}}(\frac{\sigma_i}{\sigma_{i-1}})^{t_{ij}}d\sigma_i\nonumber\\
&=\int_0^{1-\frac{\epsilon}{\sigma_{i-1}n}}\sup\limits_{j}(t_{ij}+1)y^{t_{ij}}dy\nonumber\\
&\leq \int_0^{1-\frac{\epsilon}{n}}\sup\limits_{j}(t_{ij}+1)y^{t_{ij}}dy,\label{sing_prob}
\end{alignat}
where the first inequality uses \eqref{dist_ineq}, and the equality uses the change of variable $y=\frac{\sigma_i}{\sigma_{i-1}}$. Given that the upper bound \eqref{sing_prob} is independent of $\{\sigma_j\}_{j=1}^{i-1}$, applying it to \eqref{cond_prob},  
\begin{alignat}{6}
\PP(\sigma_{i-1}-\sigma_i>\frac{\epsilon}{n})&\leq \int_0^{1-\frac{\epsilon}{n}}\sup\limits_{j}(t_{ij}+1)y^{t_{ij}}dy\int_0^1\cdot\cdot\cdot\int_0^{\sigma_{i-2}}f_{i-1}(\sigma_1,...,\sigma_{i-1})d\sigma_{i-1}...d\sigma_1\nonumber\\
&= \int_0^{1-\frac{\epsilon}{n}}\sup\limits_{j}(t_{ij}+1)y^{t_{ij}}dy.\label{final_sing_prob}
\end{alignat}
For all $j$, $t_{ij}\geq e_{\tilde{f}}(i)\geq c2^{\frac{i-1}{2}}n^2$ when $n\geq n_c$. It will now be shown that for sufficiently large $n$, $\sup\limits_j(t_{ij}+1)y^{t_{ij}}\leq (c2^{\frac{i-1}{2}}n^2+1)y^{c2^{\frac{i-1}{2}}n^2}$ for all $y\in[0,1-\frac{\epsilon}{n}]$. For a positive exponent $t>0$ and $y\in(0,1-\frac{\epsilon}{n}]$,
\begin{alignat}{6}
\frac{d}{dt}(t+1)y^{t}=y^{t}(1+(t+1)\ln(y))\leq 0 \label{neg_der}
\end{alignat}
when $t\geq - (\frac{1}{\ln(y)}+1)$. By enforcing $t\geq-(\frac{1}{\ln(1-\frac{\epsilon}{n})}+1)$, \eqref{neg_der} holds for all $y\in(0,1-\frac{\epsilon}{n}]$. Considering $t\geq cn^2$, i.e., $i=1$ in $c2^{\frac{i-1}{2}}n^2$, and $n\geq n^1_c:=\max\{\frac{1}{\epsilon c},n_c\}$, 
\begin{alignat}{6}
&&\epsilon cn&\geq 1\nonumber\\
\Rightarrow &&-\frac{\epsilon}{n}(cn^2+1)&\leq -1\nonumber\\
\Rightarrow &&e^{-\frac{\epsilon}{n}(cn^2+1)}&\leq e^{-1}\nonumber\\
\Rightarrow &&(1-\frac{\epsilon}{n})^{cn^2+1}&\leq e^{-1}\nonumber\\
\Rightarrow &&(cn^2+1)\ln(1-\frac{\epsilon}{n})&\leq -1\nonumber\\
\Rightarrow &&t\geq cn^2&\geq-\left(\frac{1}{\ln(1-\frac{\epsilon}{n})}+1\right),\label{lb_e}
\end{alignat}
where the fourth inequality uses the bound 
\begin{alignat}{6}
(1+x)^k\leq e^{xk},\label{class_ineq}
\end{alignat}
which holds for $k\geq 0$ and $x\geq -1$. For $n\geq n_c^1$, \eqref{neg_der} holds for $y\in(0,1-\frac{\epsilon}{n}]$ when $t\geq cn^2$. It follows that 
$\sup\limits_j(t_{ij}+1)y^{t_{ij}}\leq (c2^{\frac{i-1}{2}}n^2+1)y^{c2^{\frac{i-1}{2}}n^2}\leq (cn^2+1)y^{cn^2}$ for $y\in[0,1-\frac{\epsilon}{n}]$ when $n\geq n_c^1$, given that $t_{ij}\geq c2^{\frac{i-1}{2}}n^2\geq cn^2$ for all $j$ and $i\geq 1$.

Continuing from \eqref{final_sing_prob} assuming that $n\geq n^1_c$,
\begin{alignat}{6}
\PP(\sigma_{i-1}-\sigma_i>\frac{\epsilon}{n})&\leq \int_0^{1-\frac{\epsilon}{n}}(c2^{\frac{i-1}{2}}n^2+1)y^{c2^{\frac{i-1}{2}}n^2}dy\nonumber\\
&=\left(1-\frac{\epsilon}{n}\right)^{c2^{\frac{i-1}{2}}n^2+1}\nonumber\\
&\leq e^{-\frac{\epsilon}{n}(c2^{\frac{i-1}{2}}n^2+1)}\nonumber\\
&\leq e^{-\frac{\epsilon}{n}(c(1+(i-1)(\sqrt{2}-1))n^2+1)}\nonumber\\
&=e^{-\epsilon(cn+\frac{1}{n})}e^{-\epsilon cn(\sqrt{2}-1)(i-1)},\label{part_2} 
\end{alignat}
where the second inequality uses \eqref{class_ineq} and the third uses Bernoulli's inequality, $(1+x)^k\geq 1+kx$ for nonnegative integer $k$ and $x>-1$, to bound $2^{\frac{i-1}{2}}\geq 1+(i-1)(\sqrt{2}-1)$ for $i\geq 1$. Plugging \eqref{part_2} into \eqref{part_1} for $n\geq n^1_c$,

\begin{alignat}{6}
\PP(1-\sigma_n>\epsilon)&\leq e^{-\epsilon(cn+\frac{1}{n})}\sum_{i=0}^{n-1}e^{-\epsilon cn(\sqrt{2}-1)i}\nonumber\\
&= e^{-\epsilon(cn+\frac{1}{n})}\frac{1-e^{-\epsilon cn^2(\sqrt{2}-1)}}{1-e^{-\epsilon cn(\sqrt{2}-1)}}\nonumber\\
&\leq \frac{e^{-\epsilon cn}}{1-e^{-\epsilon cn(\sqrt{2}-1)}}.\nonumber
\end{alignat}

\noindent The series 
\begin{alignat}{6}
\sum_{n=n^1_c}^{\infty}\PP(1-\sigma_n>\epsilon)\leq \sum_{n=n^1_c}^{\infty}\frac{e^{-\epsilon cn}}{1-e^{-\epsilon cn(\sqrt{2}-1)}}\nonumber
\end{alignat}
is convergent for all $\epsilon>0$ using the ratio test: 
\begin{alignat}{6}
\frac{e^{-\epsilon c(n+1)}}{e^{-\epsilon cn}}\frac{1-e^{-\epsilon cn(\sqrt{2}-1)}}{1-e^{-\epsilon c(n+1)(\sqrt{2}-1)}}&\leq e^{-\epsilon c}<1,\nonumber
\end{alignat}
hence $\sigma_n\rightarrow 1$ almost surely using the first Borel-Cantelli lemma \citep[Appendix A13]{williams1991}.
\end{proof}

\noindent The previous theorem assumed that $n\rightarrow\infty$, implying that $m \rightarrow \infty$ as well. The following theorem considers the case where $m\rightarrow\infty$ for a fixed value of $n$.

\begin{theorem}
\label{theorem2}
For matrices uniformly distributed in the unit spectral norm ball, $\bm{\Delta}\sim\U(B)$, with a fixed number of rows $n$, $\lim\limits_{m\rightarrow \infty}\sigma_n=1$ almost surely.
\end{theorem}

\begin{proof}
The proof begins by redefining $e_{\tilde{f}}(i)$ for $i>1$ in equations \eqref{exp_lb_even} and \eqref{exp_lb_odd} by not assuming that $m=n$. Using the exact value of $e_1(i)$ found in Section \ref{algorithm}, for even $i$, 
\begin{alignat}{6}
e_{\tilde{f}}(i)&=(m-n)(n-i+1)+n-i-(i-1)n(n-1)+2^{\frac{i}{2}}(n^2-3n+2)\nonumber\\
&=:m(n-i+1)+z(i,n),\nonumber
\end{alignat}
and for odd $i$,
\begin{alignat}{6}
e_{\tilde{f}}(i)&=(m-n)(n-i+1)+n-i-(i-2)n(n-1)+2^{\frac{i-1}{2}}(n^2-3n+2)\nonumber\\
&=:m(n-i+1)+z(i,n),\nonumber
\end{alignat}
where the terms in $e_{\tilde{f}}(i)$ not including $m$ have been collected into the function $z$, and $e_{\tilde{f}}$ is extended to the case $i=1$ by defining $z(1,n):=-1$. Given that $\min\limits_iz(i,n)\in\ZZ$ is a fixed value independent of $m$, the function $e_{\tilde{f}}(i)\geq m+\min\limits_iz(i,n)$ can be further lower bounded by $cm$ for a constant $0<c<1$ when $m\geq m_c$ for a sufficiently large $m_c\in\ZZ$.

Wanting to again upper bound, in this case only over a finite number of terms, $\max\limits_{j}(t_{ij}+1)y^{t_{ij}}$ in \eqref{final_sing_prob}, by considering $m\geq m^n_c:=\max\{\frac{n}{\epsilon c},m_c\}$,
\begin{alignat}{6}
&&\frac{\epsilon}{n}cm&\geq 1\nonumber\\
\Rightarrow &&-\frac{\epsilon}{n}(cm+1)&\leq -1\nonumber\\
\Rightarrow &&e^{-\frac{\epsilon}{n}(cm+1)}&\leq e^{-1}\nonumber\\
\Rightarrow &&(1-\frac{\epsilon}{n})^{cm+1}&\leq e^{-1}\nonumber\\
\Rightarrow &&(cm+1)\ln(1-\frac{\epsilon}{n})&\leq -1\nonumber\\
\Rightarrow &&cm&\geq-\left(\frac{1}{\ln(1-\frac{\epsilon}{n})}+1\right),\nonumber
\end{alignat}
using the same steps to show \eqref{lb_e}. It follows that \eqref{neg_der} holds for $y\in(0,1-\frac{\epsilon}{n}]$ when $t\geq cm^n_c$, hence $\max\limits_j(t_{ij}+1)y^{t_{ij}}\leq (cm+1)y^{cm}$ for $y\in[0,1-\frac{\epsilon}{n}]$ when $m\geq m^n_c$, given that $t_{ij}\geq cm$ for all $j$ and $i\geq 1$.

Continuing from \eqref{final_sing_prob} assuming that $m\geq m^n_c$,
\begin{alignat}{6}
\PP(\sigma_{i-1}-\sigma_i>\frac{\epsilon}{n})&\leq \int_0^{1-\frac{\epsilon}{n}}(cm+1)y^{cm}dy\nonumber\\
&=\left(1-\frac{\epsilon}{n}\right)^{cm+1}\nonumber\\
&\leq e^{-\frac{\epsilon}{n}(cm+1)}.\label{part_2_2} 
\end{alignat}
Plugging \eqref{part_2_2} into \eqref{part_1}, $\PP(1-\sigma_n>\epsilon)\leq ne^{-\frac{\epsilon}{n}(cm+1)}$. By applying the ratio test,
\begin{alignat}{6}
\frac{ne^{-\frac{\epsilon}{n}(c(m+1)+1)}}{ne^{-\frac{\epsilon}{n}(cm+1)}}=e^{-\frac{\epsilon}{n}c}<1,\nonumber
\end{alignat}
\noindent it follows that the series 
\begin{alignat}{6}
\sum_{m=m^n_c}^{\infty}\PP(1-\sigma_n>\epsilon)\leq \sum_{m=m^n_c}^{\infty}ne^{-\frac{\epsilon}{n}(cm+1)}\nonumber
\end{alignat}
is convergent for all $\epsilon>0$, proving that $\sigma_n\rightarrow 1$ almost surely.    
\end{proof}

\section{Computational Experiments}
\label{Num_exp}

This section empirically verifies the convergence of $\sigma_n\rightarrow 1$, demonstrates the inability to scale up Algorithm \ref{alg:sigma}, and shows the viability of only using Algorithm \ref{alg:UV} to approximately sample from $\U(B)$. Before presenting these results, implementation details are first described.

\subsection{Implementation Details}

All computation was done using 64-bit floating-point arithmetic using Numpy 2.4.4 in Python 3.13.13. All function evaluations when sampling $\sigma_i$ are performed over a grid of $M:=1000$ evenly-spaced points $\bm{X}^i$ in the interval $[X_1,\hat{\sigma}_{i-1}]$, where the minimum value was fixed to $X_1:=0.01$: The need to take large negative exponents of $\sigma_i$ results in overflows for small values, while at the same time, the likelihood of $\sigma_i<X_1$ decreases rapidly as $n$ increases. When sampling $\sigma_i$, which will be described next, the point $X_0=0$ is also included, which still allows, through linear interpolation, sampled values of $\sigma_i\in [0,X_1]$. Let the extended set of points be denoted as $\tilde{\bm{X}}^i:=\bm{X}^i\cup \{X_0\}$.\\

\noindent\underline{Sampling $\sigma_i$}: Recalling that $\tilde{f}_i(\cdot|\sigma_1,...,\sigma_{i-1})$ is the unnormalized distribution proportional to $f_i(\cdot|\sigma_1,...,\sigma_{i-1})$, $\sigma_i$ is sampled based on the inversion method \citep[Chapter 2]{devroye1986}: An unnormalized conditional CDF is first computed given the realized samples $\{\hat{\sigma}_j\}_{j=1}^{i-1}$, 
\begin{alignat}{6}
\tilde{F}_i(X^i_j|\hat{\sigma}_1,...,\hat{\sigma}_{i-1}):=\sum_{k=1}^j\tilde{f}_i(X^i_k|\hat{\sigma}_1,...,\hat{\sigma}_{i-1}),\nonumber
\end{alignat}
\noindent with the conditional CDF of $\sigma_i$ estimated as 
\begin{alignat}{6}
F_i(X^i_j|\hat{\sigma}_1,...,\hat{\sigma}_{i-1}):=\frac{\tilde{F}_i(X^i_j|\hat{\sigma}_1,...,\hat{\sigma}_{i-1})}{\tilde{F}_i(X^i_M|\hat{\sigma}_1,...,\hat{\sigma}_{i-1})}\nonumber
\end{alignat}
\noindent over the points $\bm{X}^i$. The function $F_i:=F_i(\cdot|\hat{\sigma}_1,...,\hat{\sigma}_{i-1})$ is then extended to include $X_0$ by setting $F_i(X_0)=0$. By sampling $\hat{u}\sim \U[0,1]$, setting 
\begin{alignat}{6}
j^*=\max(\min\{j:\hat{u}\leq F_i(\tilde{X}^i_j)\},1),\nonumber
\end{alignat}
and using linear interpolation,  
\begin{alignat}{6}
\hat{\sigma}_i=\frac{((\tilde{X}^i_{j^*}-\hat{u})F_i(\tilde{X}^i_{j^*-1})+(\hat{u}-\tilde{X}^i_{j^*-1})F_i(\tilde{X}^i_{j^*}))}{(\tilde{X}^i_{j^*}-\tilde{X}^i_{j^*-1})}.\nonumber
\end{alignat}
\noindent\underline{Computing $\bm{Z}_i$}: Versus Algorithm \ref{alg:sigma}, where $\bm{Z}_i$ is updated in iteration $i$ on line \ref{rec:2}, given that $\bm{Z}_i$ must be evaluated at the points $\bm{X}^{i+2}$, which are not known until iteration $i+1$ has completed, $\bm{Z}_i$ is instead computed at the beginning of iteration $i+2$. In order to delay underflows and overflows, if 
\begin{alignat}{6}
0<Z_i^{\max}:=\max\limits_{j,k,x\in \bm{X}^{i+2}}|(Z_i(x))_{jk}|<1\quad\text{or}\quad Z_i^{\max}>\sqrt{\text{max\_value}},\nonumber
\end{alignat}
$\bm{Z}_i$ is divided by $Z_i^{\max}$ or $Z_i^{\max}/\sqrt{\text{max\_value}}$, respectively, where $\text{max\_value}$ is the maximum representable number in the 64-bit floating-point representation, with $\sqrt{\text{max\_value}}=1.341\times10^{154}$. As a final detail, there were only two $\bm{Z}_i$ tensors stored, $\bm{Z}_e$ and $\bm{Z}_o$ for even and odd indexed singular values, respectively, with one being updated at the beginning of each iteration.\\ 

\noindent\underline{Exponent Canceling}: The unnormalized distribution $\tilde{f}_i(\cdot|\sigma_1,...,\sigma_{i-1})$ consists of the product of $\sigma_i^{e_1(i)+e_2(i)}$ and $|d_i(\sigma_i^2)|$. As given by $2D(i)$ in the proof of Theorem \ref{theorem1}, the exponents of $\sigma_i$ in $|d_i(\sigma_i^2)|$ are non-negative and increasing with $n$ and $i$. The exponent $e_1(i)+e_2(i)$, found in Section \ref{algorithm}, is always decreasing in $i$, and for values of $m<2n$, $e_1(i)+e_2(i)$ is negative: 
\begin{alignat}{6}
e_1(i)+e_2(i)&\leq (m-n)(n-1)+n-2-n(n-1)\nonumber\\
             &\leq (n-1)^2+n-2-n(n-1)\nonumber\\
             &=-1,\nonumber
\end{alignat}
where the first inequality considers $i=2$, and the second inequality uses $m=2n-1$.

For values of $\sigma_i<1$, as $n$ increases, large negative exponents $e_1(i)+e_2(i)$ cause overflows, whereas large positive exponents within $|d_i(\sigma_i^2)|$ cause underflows, which can be remedied by shifting negative exponents from $e_1(i)+e_2(i)$ to $|d_i(\sigma_i^2)|$. 

For $i\in\{2,3\}$, $e_1(i)+e_2(i)$ is initially increased by a quantity $e_3(2)=e_3(3)>0$. To correct for this, within $\bm{Z}_0$ and $\bm{Z}_1$, $\beta$ is replaced by $\beta-\frac{e_3(2)}{4}$, resulting in all of the polynomial exponents within $\bm{Z}_0(\sigma_2^2)$ and $\bm{Z}_1(\sigma_3^2)$, and subsequently within $d_2(\sigma_2^2)$ and $d_3(\sigma_3^2)$, to decrease by $e_3(2)$. This initial injection of $-e_3(2)$ decreases exponentially given the recursive updates of $\bm{Z}_i$ on line \ref{rec:2} of Algorithm \ref{alg:sigma}. At iteration $i$, the exponent of $\sigma_i$ in $d_i(\sigma_i^2)$ has decreased by $-e_3(2)2^{\lfloor{\frac{i}{2}\rfloor}-1}$. To correct for this, for $i>3$ $e_1(i)+e_2(i)$ is increased by $e_3(i):=e_3(2)2^{\lfloor{\frac{i}{2}\rfloor}-1}$.

The initial quantity $e_3(2)$ is set to 
\begin{alignat}{6}
e_3(2):&=\frac{\min(\max(-(e_1(n)+e_2(n)),0),2D(n))}{2^{\lfloor{\frac{n}{2}\rfloor}-1}}. \label{ec_formula}
\end{alignat}
When $e_1(n)+e_2(n)<0$ and $-(e_1(n)+e_2(n))\leq 2D(n)$, 
$e_1(n)+e_2(n)+e_3(n)=0$, meaning that by the $n^{th}$ iteration the negative exponent $e_1(n)+e_2(n)$ has been cancelled out with the positive exponents of $\sigma_n$ within $|d_n(\sigma_n^2)|$, and likewise, when $-(e_1(n)+e_2(n))>2D(n)$, the positive exponents of $\sigma_n$ within $|d_n(\sigma_n^2)|$ are cancelled out with the negative exponent $e_1(n)+e_2(n)$. 

In addition to using exponent canceling, the convention that $0*\infty=0$ is applied when computing $\tilde{f}_i(\cdot|\hat{\sigma}_1,...,\hat{\sigma}_{i-1})$ over $\bm{X}^i$ by setting $\tilde{f}_i(x|\hat{\sigma}_1,...,\hat{\sigma}_{i-1})=0$ when $d_i(x^2)=0$, regardless of $x^{e_1(i)+e_2(i)}$ overflowing or not.

\subsection{Results}

With the described implementation, for each choice of $(n,m)$, $n$ singular values were generated 1000 times. Beginning with square matrices, it was only possible to generate up to $n=19$ singular values. The experiment with $n=m=20$ was stopped after the computed conditional CDF underflowed for all $x\in \bm{X}^{20}$ in the $99^{th}$ trial. Fixing $n=19$, singular values were then generated for $m\in\{50,100,250,500,1000\}$. The mean and standard deviation of all sampled singular values are given in Table \ref{tab:1}. In each trial, samples of $\bm{U}$ and $\bm{V}$ were also generated using Algorithm \ref{alg:UV}, with the computed errors 
\begin{alignat}{6}
\text{ave error}&:=\frac{1}{nm}\sum_{j=1}^n\sum_{k=1}^m|(\hat{\bm{U}}\hat{\bm{\Sigma}}\hat{\bm{V}}^T-\hat{\bm{U}}\hat{\bm{V}}^T)_{jk}|\quad\text{and}\label{mat_errors}\\
\text{max error}&:=\max\limits_{j,k}|(\hat{\bm{U}}\hat{\bm{\Sigma}}\hat{\bm{V}}^T-\hat{\bm{U}}\hat{\bm{V}}^T)_{jk}|\nonumber
\end{alignat}
\noindent from replacing $\hat{\bm{\Sigma}}:=\diag(\hat{\sigma}_1,..,\hat{\sigma}_n)$ with an identity matrix, bypassing the use of Algorithm 

\begin{table}[H]
\centering
\resizebox{1.0\textwidth}{!}{
\setlength{\tabcolsep}{4pt}
	\begin{tabular}{ccccccccccccccccccccccc}		
	\hline
    n & m & & $\hat{\sigma}_1$ & $\hat{\sigma}_2$ & $\hat{\sigma}_3$ & $\hat{\sigma}_4$ & $\hat{\sigma}_5$ & $\hat{\sigma}_6$ & $\hat{\sigma}_7$ & $\hat{\sigma}_8$ & $\hat{\sigma}_9$ & $\hat{\sigma}_{10}$ & $\hat{\sigma}_{11}$ & $\hat{\sigma}_{12}$ & $\hat{\sigma}_{13}$ & $\hat{\sigma}_{14}$ & $\hat{\sigma}_{15}$ & $\hat{\sigma}_{16}$ & $\hat{\sigma}_{17}$ & $\hat{\sigma}_{18}$ & $\hat{\sigma}_{19}$\\
    \hline
	\multirow{2}{*}{1}& \multirow{2}{*}{1}&$\mu$ & 0.509 &  &  &  &  &  &  &  &  &  &  &  &  & &  &  &  &  &  \\ 
                      &  & $s$  & 0.287 &  &  &  &  &  &  &  &  &  &  &  &  &  & &  &  &  & \\		 	
    \hline			
    \multirow{2}{*}{2}& \multirow{2}{*}{2}& $\mu$  & 0.803 & 0.299 &  &  &  &  &  &  &  &  &  &  &  & &  &  &  & & \\ 
                      & & $s$   & 0.159 & 0.203 &  &  &  &  &  &  &  &  &  &  &  & &  &  &  & & \\		 	
    \hline		
    \multirow{2}{*}{3}& \multirow{2}{*}{3}& $\mu$  & 0.902 & 0.604 & 0.210 &  &  &  &  &  &  &  &  &  & &  &  &  & &  & \\ 
                      & & $s$   & 0.087 & 0.168 & 0.146 &  &  &  &  &  &  &  &  &  & &  &  &  &  &  & \\		 	
    \hline		
    \multirow{2}{*}{4}&  \multirow{2}{*}{4}& $\mu$  & 0.942 & 0.759 & 0.478 & 0.164 &  &  &  &  &  &  & &  &  &  & &  &  &  & \\ 
                      & & $s$   & 0.053 & 0.115 & 0.143 & 0.113 &  &  &  &  &  &  &  &  & &  &  &  & &  & \\		 	
    \hline		
    \multirow{2}{*}{5}& \multirow{2}{*}{5}& $\mu$  & 0.962 & 0.840 & 0.644 & 0.394 & 0.052 &  &  &  &  &  &  &  &  & &  &  &  & & \\ 
                      & & $s$   & 0.035 & 0.081 & 0.111 & 0.125 & 0.070 &  &  &  &  &  &  &  &  & &  &  &  & & \\		 	
    \hline		
    \multirow{2}{*}{6}&  \multirow{2}{*}{6}& $\mu$  & 0.974 & 0.887 & 0.745 & 0.556 & 0.325 & 0.309 &  &  &  &  &  &  &  & &  &  &  & &  \\ 
                      & & $s$  & 0.025 & 0.059 & 0.085 & 0.105 & 0.125 & 0.118 &  &  &  &  &  &  &  & &  &  &  & &  \\		 	
    \hline		
    \multirow{2}{*}{7}&\multirow{2}{*}{7}&$\mu$  & 0.980 & 0.916 & 0.809 & 0.664 & 0.489 & 0.472 & 0.452 &  &  &  &  &  & &  &  &  & &  & \\ 
                      &  & $s$  & 0.019 & 0.044 & 0.066 & 0.085 & 0.101 & 0.097 & 0.092 &  &  &  &  &  & &  &  &  & &  & \\		 	
    \hline
    \multirow{2}{*}{8}&\multirow{2}{*}{8}&$\mu$  & 0.985 & 0.935 & 0.852 & 0.737 & 0.597 & 0.582 & 0.564 & 0.562 &  &  &  &  & &  &  &  & &  & \\ 
                      & & $s$   & 0.014 & 0.035 & 0.052 & 0.070 & 0.086 & 0.083 & 0.081 & 0.080 &  &  &  &  & &  &  &  & &  & \\		 	
    \hline
    \multirow{2}{*}{9}& \multirow{2}{*}{9}&$\mu$  & 0.988 & 0.948 & 0.882 & 0.789 & 0.675 & 0.662 & 0.649 & 0.647 & 0.645 &  &  & &  &  &  & &  &  & \\ 
                      & & $s$  & 0.011 & 0.028 & 0.042 & 0.057 & 0.072 & 0.070 & 0.071 & 0.071 & 0.071 &  &  & &  &  &  & &  &  & \\	
    \hline
     \multirow{2}{*}{10}&\multirow{2}{*}{10}&$\mu$  & 0.990 & 0.958 & 0.903 & 0.828 & 0.735 & 0.725 & 0.718 & 0.717 & 0.715 & 0.714 &  &  & &  &  &  & &  & \\ 
                      & & $s$   & 0.009 & 0.023 & 0.035 & 0.049 & 0.065 & 0.065 & 0.065 & 0.065 & 0.065 & 0.065 &  &  &  & &  &  &  & &  \\		 	
    \hline
     \multirow{2}{*}{11}& \multirow{2}{*}{11}& $\mu$  & 0.992 & 0.965 & 0.920 & 0.884 & 0.839 & 0.834 & 0.828 & 0.827 & 0.825 & 0.824 & 0.823 &  & &  &  &  & &  & \\ 
                      & & $s$  & 0.008 & 0.019 & 0.029 & 0.050 & 0.078 & 0.078 & 0.078 & 0.078 & 0.077 & 0.077 & 0.077 &  &  & &  &  &  & &  \\		 	
    \hline
     \multirow{2}{*}{12}&  \multirow{2}{*}{12}& $\mu$  & 0.993 & 0.970 & 0.932 & 0.919 & 0.900 & 0.896 & 0.892 & 0.890 & 0.889 & 0.888 & 0.887 & 0.886  &  & &  &  &  & &  \\ 
                      &  & $s$  & 0.006 & 0.016 & 0.025 & 0.027 & 0.034 & 0.034 & 0.034 & 0.034 & 0.034 & 0.034 & 0.034 & 0.034  &  & &  &  &  & & \\		 	
    \hline
     \multirow{2}{*}{13}&  \multirow{2}{*}{13}& $\mu$  & 0.994 & 0.975 & 0.942 & 0.931 & 0.917 & 0.913 & 0.910 & 0.908 & 0.907 & 0.906 & 0.905 & 0.904 & 0.904  & &  &  &  & & \\ 
                      & & $s$  & 0.006 & 0.014 & 0.021 & 0.023 & 0.026 & 0.026 & 0.026 & 0.026 & 0.026 & 0.026 & 0.026 & 0.026 & 0.026  & &  &  &  & & \\		 	
    \hline
     \multirow{2}{*}{14}& \multirow{2}{*}{14}&$\mu$  & 0.995 & 0.980 & 0.952 & 0.942 & 0.930 & 0.927 & 0.924 & 0.923 & 0.922 & 0.921 & 0.920 & 0.919 & 0.919 & 0.918 &  &  &  & & \\ 
                      & & $s$  & 0.005 & 0.011 & 0.018 & 0.020 & 0.022 & 0.022 & 0.023 & 0.022 & 0.023 & 0.022 & 0.022 & 0.023 & 0.023 & 0.023  & &  &  &  & \\		 	
    \hline
     \multirow{2}{*}{15}&  \multirow{2}{*}{15}& $\mu$  & 0.996 & 0.982 & 0.960 & 0.951 & 0.942 & 0.939 & 0.936 & 0.935 & 0.934 & 0.933 & 0.933 & 0.932 & 0.931 & 0.931 & 0.930 & &  &  &  \\ 
                      &  & $s$   & 0.004 & 0.010 & 0.015 & 0.017 & 0.019 & 0.019 & 0.019 & 0.019 & 0.019 & 0.019 & 0.019 & 0.019 & 0.019 & 0.019 & 0.019 &  & &  & \\	
    \hline
     \multirow{2}{*}{16}& \multirow{2}{*}{16}&  $\mu$  & 0.996 & 0.984 & 0.964 & 0.957 & 0.949 & 0.947 & 0.945 & 0.944 & 0.943 & 0.942 & 0.941 & 0.941 & 0.940 & 0.940 & 0.939 & 0.939 &  &  &  \\ 
                      & & $s$   & 0.004 & 0.009 & 0.015 & 0.016 & 0.018 & 0.018 & 0.018 & 0.018 & 0.018 & 0.018 & 0.018 & 0.018 & 0.018 & 0.018 & 0.018 & 0.018 &  &  & \\		 	
    \hline
     \multirow{2}{*}{17}&\multirow{2}{*}{17}&$\mu$  & 0.997 & 0.987 & 0.972 & 0.966 & 0.959 & 0.957 & 0.954 & 0.953 & 0.952 & 0.952 & 0.951 & 0.95 & 0.95 & 0.949 & 0.949 & 0.948 & 0.948  &  & \\ 
                      & & $s$   & 0.003 &  0.008 & 0.011 & 0.012 & 0.014 & 0.014 & 0.014 & 0.014 & 0.014 & 0.014 & 0.014 & 0.014 & 0.014 & 0.014 & 0.014 & 0.014 & 0.014 &  & \\		 	
    \hline
     \multirow{2}{*}{18}&\multirow{2}{*}{18}&$\mu$  & 0.997 & 0.988 & 0.975 & 0.970 & 0.964 & 0.962 & 0.960 & 0.959 & 0.958 & 0.957 & 0.957 & 0.956 & 0.956 & 0.955 & 0.955 & 0.954 & 0.954 & 0.953 & \\ 
                      & & $s$   & 0.003 &  0.007 & 0.012 & 0.012 & 0.014 & 0.014 & 0.014 & 0.014 & 0.014 & 0.014 & 0.014 & 0.014 & 0.014 & 0.014 & 0.014 & 0.014 & 0.014 & 0.014 & \\		 	
     \hline
     \multirow{2}{*}{19}& \multirow{2}{*}{19}&$\mu$  & 0.997 & 0.989 & 0.976 & 0.972 & 0.966 & 0.965 & 0.963 & 0.962 & 0.961 & 0.961 & 0.960 & 0.959 & 0.959 & 0.958 & 0.958 & 0.958 & 0.957 & 0.957 &  0.956 \\ 
                       & & $s$   & 0.003 & 0.006 & 0.009 & 0.01 & 0.011 & 0.011 & 0.011 & 0.011 & 0.011 & 0.011 & 0.011 & 0.011 & 0.011 & 0.011 & 0.011 & 0.011 & 0.011 & 0.011 & 0.011 \\		 	                  
    \hline   
    \multirow{2}{*}{19}& \multirow{2}{*}{50}&$\mu$  &  0.999 & 0.997 & 0.996 & 0.994 & 0.992 & 0.990 & 0.989 & 0.988 & 0.987 & 0.987 & 0.986 & 0.985 & 0.985 & 0.984 & 0.984 & 0.983 & 0.983 & 0.982 & 0.982 \\ 
                       &                    & $s$   &  0.001 & 0.002 & 0.002 & 0.003 & 0.003 & 0.004 & 0.004 & 0.004 & 0.004 & 0.004 & 0.004 & 0.004 & 0.004 & 0.004 & 0.004 & 0.004 & 0.004 & 0.004 & 0.004\\		 	                  
    \hline 
    \multirow{2}{*}{19}& \multirow{2}{*}{100}&$\mu$  &  0.999 & 0.998 & 0.998 & 0.997 & 0.996 & 0.995 & 0.994 & 0.993 & 0.992 & 0.991 & 0.991 & 0.990 & 0.989 & 0.989 & 0.988 & 0.988 & 0.987 & 0.987 & 0.986\\ 
                       &                     & $s$   &  0.001 & 0.001 & 0.001 & 0.001 & 0.002 & 0.002 & 0.002 & 0.002 & 0.002 & 0.002 & 0.002 & 0.002 & 0.002 & 0.003 & 0.003 & 0.003 & 0.003 & 0.003 & 0.003\\		 	                  
    \hline 
    \multirow{2}{*}{19}& \multirow{2}{*}{250}&$\mu$  &  1.000 & 0.999 & 0.998 & 0.998 & 0.997 & 0.997 & 0.996 & 0.995 & 0.995 & 0.994 & 0.993 & 0.993 & 0.992 & 0.992 & 0.991 & 0.991 & 0.990 & 0.990 & 0.989\\ 
                       &                     & $s$   &  0.000 & 0.000 & 0.001 & 0.001 & 0.001 & 0.001 & 0.001 & 0.001 & 0.001 & 0.001 & 0.001 & 0.001 & 0.002 & 0.002 & 0.002 & 0.002 & 0.002 & 0.002 & 0.002\\		 	                  
    \hline 
    \multirow{2}{*}{19}& \multirow{2}{*}{500}&$\mu$  & 1.000 & 0.999 & 0.999 & 0.998 & 0.998 & 0.997 & 0.996 & 0.996 & 0.995 & 0.995 & 0.994 & 0.994 & 0.993 & 0.993 & 0.992 & 0.992 & 0.991 & 0.991 & 0.990 \\ 
                       &                     & $s$   & 0.000 & 0.000 & 0.000 & 0.001 & 0.001 & 0.001 & 0.001 & 0.001 & 0.001 & 0.001 & 0.001 & 0.001 & 0.001 & 0.001 & 0.001 & 0.001 & 0.001 & 0.001 & 0.001 \\		 	                  
    \hline 
    \multirow{2}{*}{19}& \multirow{2}{*}{1000}&$\mu$  & 1.000 & 0.999 & 0.999 & 0.998 & 0.998 & 0.997 & 0.997 & 0.996 &  0.996 & 0.995 & 0.995 & 0.994 & 0.994 & 0.993 & 0.993 & 0.992 & 0.992 & 0.991 & 0.991 \\ 
                       &                      & $s$   & 0.000 & 0.000 & 0.000 & 0.001 & 0.001 & 0.001 & 0.001 & 0.001 &  0.001 & 0.001 & 0.001 & 0.001 & 0.001 & 0.001 & 0.001 & 0.001 & 0.001 & 0.001 & 0.001\\	
    \hline     
	\end{tabular}        
    }
	\caption{The mean $\mu$ and standard deviation $s$ from 1000 samples of the singular values $\{\sigma_i\}$ of uniformly sampled $n\times m$ matrices from the unit spectral norm ball $B$ using Algorithm \ref{alg:sigma}.}
 \label{tab:1}
\end{table}

\noindent \ref{alg:sigma}. The mean and standard deviation of these errors over the 1000 trials are given in Table \ref{tab:2}. The values of $e_3(2)$, $e_1(n)$, $e_2(n)$, and $2D(n)$ are also given, which were used for exponent canceling following \eqref{ec_formula}. 

Examining Table \ref{tab:1}, starting from $n=m=6$, $\mu(\hat{\sigma}_n)$ is monotonically increasing with $\mu(\hat{\sigma}_{19})=0.956$ when $n=m=19$. This is interpreted as verification that $\sigma_n\rightarrow 1$ following Theorem \ref{theorem1}. From the starting point of $n=m=19$, by further increasing $m$, $\mu(\hat{\sigma}_n)$ continues to monotonically increase to a final value of $\mu(\hat{\sigma}_n)=0.991$ when $m=1000$, verifying Theorem \ref{theorem2}. 

\begin{table}[H]
\centering
\scriptsize
\setlength{\tabcolsep}{2pt}
	\begin{tabular}{ccccccccc||ccccccccc}		
	\hline
    n & m &$e_3(2)$ & $e_1(n)$&$e_2(n)$&$2D(n)$& & ave error & max error & n & m &$e_3(2)$ & $e_1(n)$&$e_2(n)$&$2D(n)$& & ave error & max error\\    
    \hline    
	\multirow{2}{*}{1}& \multirow{2}{*}{1}& \multirow{2}{*}{-} &\multirow{2}{*}{-}&\multirow{2}{*}{-}&\multirow{2}{*}{-} & $\mu$ & 0.491 & 0.491 &   \multirow{2}{*}{13}&  \multirow{2}{*}{13}& \multirow{2}{*}{53.625} & \multirow{2}{*}{0} &\multirow{2}{*}{-1716}&\multirow{2}{*}{8,448} &$\mu$ & 0.018 & 0.060  \\  
                      &                   &                    &                  &                  & & $s$   & 0.287 & 0.287 &                      &                     &                        &                       &&& $s$  & 0.005  & 0.018    \\
    \hline		
    \multirow{2}{*}{2}&  \multirow{2}{*}{2}& \multirow{2}{*}{0} & \multirow{2}{*}{0} &\multirow{2}{*}{-2} &\multirow{2}{*}{0} & $\mu$ & 0.316 & 0.574 &   \multirow{2}{*}{14}& \multirow{2}{*}{14}& \multirow{2}{*}{36.969} & \multirow{2}{*}{0} &\multirow{2}{*}{-2366}&\multirow{2}{*}{19,968} &$\mu$ & 0.015 & 0.050 \\
                      &                    &                    &                    &                    &  & $s$   & 0.100 & 0.179 &                      &                    &                        &                     &  && $s$   & 0.004 & 0.015  \\
    \hline
    \multirow{2}{*}{3}& \multirow{2}{*}{3}&\multirow{2}{*}{4} & \multirow{2}{*}{0} &\multirow{2}{*}{-6}&\multirow{2}{*}{4} &$\mu$ & 0.246 & 0.573   &  \multirow{2}{*}{15}& \multirow{2}{*}{15}&  \multirow{2}{*}{42.656} & \multirow{2}{*}{0} &\multirow{2}{*}{-2730}&\multirow{2}{*}{23,296} &$\mu$ & 0.012 & 0.042 \\  
                      &                   &                   &                    &                   & &$s$   & 0.054 & 0.136   &                     &                    &                         &                       &&& $s$ & 0.004 & 0.012 \\
    \hline
    \multirow{2}{*}{4}& \multirow{2}{*}{4}& \multirow{2}{*}{12} & \multirow{2}{*}{0} &\multirow{2}{*}{-36}&\multirow{2}{*}{24} &$\mu$ & 0.207 & 0.552 & \multirow{2}{*}{16}& \multirow{2}{*}{16}& \multirow{2}{*}{28.125} & \multirow{2}{*}{0} &\multirow{2}{*}{-3600}&\multirow{2}{*}{53,760} &$\mu$ & 0.011 & 0.036  \\	
                      &                   &                     &                     &                   & & $s$   & 0.034 & 0.108 &                    &                    &                        &                         &&& $s$ & 0.003 & 0.011 \\
    \hline
    \multirow{2}{*}{5}&  \multirow{2}{*}{5}& \multirow{2}{*}{24} & \multirow{2}{*}{0} & \multirow{2}{*}{-60}&\multirow{2}{*}{48}  &$\mu$ & 0.191 & 0.569 & \multirow{2}{*}{17}&  \multirow{2}{*}{17}& \multirow{2}{*}{31.875} & \multirow{2}{*}{0} &\multirow{2}{*}{-4080}&\multirow{2}{*}{61,440} &$\mu$ & 0.009 & 0.030\\ 
                      &                    &                     &                     &                    & & $s$  & 0.019 & 0.092 &	                 &                     &                        &                       &&& $s$ & 0.002 & 0.009   \\	
    \hline	
    \multirow{2}{*}{6}& \multirow{2}{*}{6}& \multirow{2}{*}{37.5} & \multirow{2}{*}{0} &\multirow{2}{*}{-150}&\multirow{2}{*}{160}  &$\mu$ & 0.148 & 0.437 & \multirow{2}{*}{18}& \multirow{2}{*}{18}& \multirow{2}{*}{20.320} & \multirow{2}{*}{0} &\multirow{2}{*}{-5202}&\multirow{2}{*}{139,264} &$\mu$ & 0.007 & 0.027\\ 
                      &                   &                       &                      &                   & &$s$   & 0.025 & 0.093 &                    &                    &                        &                       &&& $s$ & 0.002 & 0.009 \\	
    \hline    
    \multirow{2}{*}{7}& \multirow{2}{*}{7}& \multirow{2}{*}{52.5} & \multirow{2}{*}{0} &\multirow{2}{*}{-210}&\multirow{2}{*}{240}  &$\mu$ & 0.116 & 0.346 &  \multirow{2}{*}{19}&  \multirow{2}{*}{19}& \multirow{2}{*}{22.711} & \multirow{2}{*}{0} &\multirow{2}{*}{-5814}&\multirow{2}{*}{156,672} &$\mu$ & 0.007 & 0.024\\
                      &                   &                       &                      &                   & & $s$  & 0.021 & 0.075 &		                 &                     &                        &                     &  && $s$  & 0.002 & 0.007\\ 	
    \hline
    \multirow{2}{*}{8}&  \multirow{2}{*}{8}& \multirow{2}{*}{49} & \multirow{2}{*}{0} &\multirow{2}{*}{-392} &\multirow{2}{*}{672}  &$\mu$ & 0.091 & 0.280 &  \multirow{2}{*}{19}& \multirow{2}{*}{50} & \multirow{2}{*}{22.590} & \multirow{2}{*}{31}  &\multirow{2}{*}{-5814}&\multirow{2}{*}{156,672} &$\mu$  & 0.001 & 0.006 \\
                      &                    &                     &                      &                    & &$s$  & 0.017 & 0.063 &                     &                     &                         &                      &&& $s$    & 0.000 & 0.002 \\	
    \hline    
    \multirow{2}{*}{9}& \multirow{2}{*}{9}& \multirow{2}{*}{63} & \multirow{2}{*}{0} &\multirow{2}{*}{-504} &\multirow{2}{*}{896} &$\mu$ & 0.073 & 0.227 &   \multirow{2}{*}{19}& \multirow{2}{*}{100} & \multirow{2}{*}{22.395} & \multirow{2}{*}{81}  &\multirow{2}{*}{-5814}&\multirow{2}{*}{156,672} &$\mu$  & 0.001 & 0.003  \\ 
                      &                   &                     &                      &                    &  & $s$  & 0.015 & 0.050 &	                   &                      &                         &                    &&& $s$       & 0.000 & 0.001 \\		     	 	    
    \hline
    \multirow{2}{*}{10}& \multirow{2}{*}{10}& \multirow{2}{*}{50.625} & \multirow{2}{*}{0} &\multirow{2}{*}{-810}&\multirow{2}{*}{2,304} &$\mu$ & 0.058 & 0.182 &   \multirow{2}{*}{19}& \multirow{2}{*}{250} & \multirow{2}{*}{21.809} & \multirow{2}{*}{231}  &\multirow{2}{*}{-5814}&\multirow{2}{*}{156,672} &$\mu$  & 0.000 & 0.002 \\ 
                       &                    &                        &                      &                    & & $s$   & 0.013 & 0.045 &                      &                      &                         &                  &   && $s$      & 0.000 & 0.000 \\              
    \hline        
     \multirow{2}{*}{11}& \multirow{2}{*}{11}& \multirow{2}{*}{61.875} & \multirow{2}{*}{0} & \multirow{2}{*}{-990}&\multirow{2}{*}{2,880} &$\mu$ & 0.035 & 0.112 &  \multirow{2}{*}{19}& \multirow{2}{*}{500} & \multirow{2}{*}{20.832} & \multirow{2}{*}{481}  &\multirow{2}{*}{-5814}&\multirow{2}{*}{156,672} &$\mu$  & 0.000 & 0.001 \\ 
                        &                    &                        &                      &                     & & $s$  & 0.015 & 0.051  &	                   &                      &                         &                  &    && $s$    & 0.000 & 0.000 \\              
    \hline
    \multirow{2}{*}{12}& \multirow{2}{*}{12}& \multirow{2}{*}{45.375} & \multirow{2}{*}{0} &\multirow{2}{*}{-1452}&\multirow{2}{*}{7,040} &$\mu$ & 0.022 & 0.072 & \multirow{2}{*}{19}& \multirow{2}{*}{1000} & \multirow{2}{*}{18.879} & \multirow{2}{*}{981}  &\multirow{2}{*}{-5814}&\multirow{2}{*}{156,672} &$\mu$  &  0.000 & 0.001 \\ 
                       &                    &                        &                       &                    & & $s$   & 0.007 & 0.023 &                    &                       &                         &                    &&& $s$      & 0.000 & 0.000 \\
    \hline
	\end{tabular}        
	\caption{For each matrix size $(n>1)\times m$, the value of $e_3(2)$ used for exponent canceling following \eqref{ec_formula} is shown given the values of $e_1(n)$ and $e_2(n)$ in Algorithm \ref{alg:sigma} and $2D(n)$ from the proof of Theorem \ref{theorem1}. The columns ave and max error contain the mean $\mu$ and standard deviation $s$ from 1000 samples of computing these errors following \eqref{mat_errors}.}
 \label{tab:2}
\end{table}
\noindent Examining Table \ref{tab:2}, beginning again from $n=m=6$, the means of ave and max error are monotonically decreasing with 
$\mu(\text{ave error})=0.007$ and $\mu(\text{max error})=0.024$ when $n=m=19$. By further increasing $m$, these mean errors continue to decrease to $0.001$ and $0.006$ when $m=50$, and finally to $0.000$ and $0.001$ when $m=1000$, demonstrating the viability in terms of accuracy of only using Algorithm \ref{alg:UV} to generate samples $\hat{\bm{\Delta}}$. The remaining columns show the use of \eqref{ec_formula} in choosing $e_3(2)$, and perhaps more interestingly, the scale of the exponents required when only sampling up to 19 singular values. The lower bound on the exponents contained in $|d_n(\sigma_n^2)|$, $2D(n)=O(n^22^{\frac{n}{2}})$, hence its value will continue to grow exponentially, demonstrating the futility of trying to sample $\{\sigma_i\}$ for larger values of $n$. In addition, $|e_1(n)+e_2(n)|=O(n^3)$, limiting the benefit of exponent canceling as $n$ increases, which can already be seen when $n=m=19$, where $e_1(n)+e_2(n)=-5,814$ is already much smaller than $2D(n)=156,672$.

As a final experiment, Algorithm \ref{alg:UV} was tested on the example matrix size of $5,120\times 25,600$ given in Section \ref{intro}, and it was verified that approximate samples of $\hat{\bm{\Delta}}=\hat{\bm{U}}\hat{\bm{V}}^T$ can be generated using numpy.linalg.qr \citep{npqr} with 64 and 32-bit floating-point arithmetic.

\section{Conclusion}
This paper studied uniform sampling from the unit spectral norm ball in high dimension. Motivated by numerical challenges of directly sampling from this distribution, it was proven that as the size of the sampled matrix increases, all singular values converge to 1 almost surely. This finding led to a much simpler approximate sampling algorithm which is scalable and shown empirically to have monotonically decreasing error as the size of the matrix increases. 

\bibliography{bibliography_paper}

\end{document}